\documentclass[12pt]{amsart}
\usepackage[dvips]{graphicx,color}
\usepackage[dvips]{epsfig}
\usepackage{amsmath,amssymb,latexsym,amscd,xypic}
\DeclareRobustCommand{\SkipTocEntry}[4]{}
\makeatletter
\newcommand\@dotsep{4.5}
\def\@tocline#1#2#3#4#5#6#7{\relax
  \ifnum #1>\c@tocdepth % then omit
  \else
    \par \addpenalty\@secpenalty\addvspace{#2}%
    \begingroup \hyphenpenalty\@M
    \@ifempty{#4}{%
      \@tempdima\csname r@tocindent\number#1\endcsname\relax
    }{%
      \@tempdima#4\relax
    }%
    \parindent\z@ \leftskip#3\relax \advance\leftskip\@tempdima\relax
    \rightskip\@pnumwidth plus1em \parfillskip-\@pnumwidth
    #5\leavevmode\hskip-\@tempdima #6\relax
    \leaders\hbox{$\m@th
      \mkern \@dotsep mu\hbox{.}\mkern \@dotsep mu$}\hfill
    \hbox to\@pnumwidth{\@tocpagenum{#7}}\par
    \nobreak
    \endgroup
  \fi}
\makeatother 

\DeclareFontFamily{OT1}{rsfs}{}
\DeclareFontShape{OT1}{rsfs}{n}{it}{<-> rsfs10}{}
\DeclareMathAlphabet{\curly}{OT1}{rsfs}{n}{it}

\newcommand{\cO}{\mathcal{O}}

\newcommand{\T}{\mathbf{T}}
\newcommand\C{\mathbb C}
\newcommand\CC{\mathsf C}
\newcommand\JJ{\mathsf P}

\newcommand\MM{\mathsf{M}}

\newcommand\OO{\mathcal{O}}
\newcommand\PP{\mathbb P}
\newcommand\Pp{{\mathbb P}^1}
\newcommand\FF{\mathbb F}
\newcommand\FFF{\mathsf{F}}

\newcommand\Q{\mathbb Q}

\newcommand\Z{\mathbb Z}
\newcommand\bW{\mathsf{W}}

\newcommand\ZZ{\mathsf Z}
\newcommand\com{\mathbb C}

\newcommand{\bV}{\mathsf{V}}

\newcommand{\bF}{\mathsf{F}}

\newcommand\Into{\ar@{^{ (}->}[r]}

\newcommand\ch{\operatorname{ch}}

\newcommand\beq{\begin{equation}}
\newcommand\eeq{\end{equation}}
\newtheorem{thm}{Theorem}

\newtheorem{lem}{Lemma}

\newtheorem{prop}{Proposition}

\title{\textbf{Descendents on local curves:\\ Stationary theory}}
\author{R. Pandharipande and A. Pixton}
\date{September 2011}

\begin{document}

\begin{abstract} \noindent
The stable pairs theory of local curves in 3-folds 
(equivariant with respect to the scaling 2-torus)
is studied 
with stationary descendent insertions.
Reduction rules are found to lower descendents
when higher than the degree. Factorization 
then yields a simple proof of rationality in
the stationary case and a proof
of the functional equation related to inverting $q$. The
method yields an
effective determination of stationary descendent
integrals.
The series $\mathsf{Z}^{\mathsf{cap}}_{d,(d)}( \tau_d(\mathsf{p}))$
 plays a special role and is calculated exactly using
the stable pairs vertex and an analysis
of the solution of the  quantum differential 
equation for the Hilbert scheme of points of the plane.
\end{abstract}

\maketitle

\setcounter{tocdepth}{1} 
\tableofcontents

%%%%%%%%%%%%%%%%%%%%%%%%%%%%%%%%%%%%%%%%%%%%%%%%%%%%%%%%%%%%%%%%%%%%%%%%%%%

\setcounter{section}{-1}
\section{Introduction}

\subsection{Relative local curves} \label{lc1}
The geometry of a 3-fold {\em local curve} consists of a  
 split 
rank 2 bundle $N$ on a nonsingular projective curve $C$ of genus $g$,
\begin{equation}\label{ffg}
N=L_1\oplus L_2.
\end{equation}
The splitting determines a scaling action of a 2-dimensional torus
$$T=\C^* \times \C^*$$ on $N$.
The {\em level} of the splitting is the pair of integers
$(k_1,k_2)$ where,
$$k_i= {\text {deg}}(L_i).$$
Of course, the scaling action and the level
depend upon the choice of splitting \eqref{ffg}.

The fiber of $N$ over a point $p\in C$ determines a $T$-invariant
divisor 
$$N_p \subset N$$
isomorphic to $\com^2$ with the standard $T$-action.
We will consider the local stable pairs theory of $N$
relative to the divisor
$$S= \bigcup_{i=1}^r N_{p_i} \subset N$$
determined by the fibers over $p_1,\ldots,p_r\in C$.
Let $P_n(N/S,d)$ denote the relative moduli space of stable pairs{\footnote
{The curve class is $d$ times the zero section $C \subset N$.}}, see 
\cite{pt}.

For each $p_i$, let $\eta^i$
be a partition of $d$ weighted
by the equivariant Chow ring, 
$$A_T^*(N_{p_i},{\mathbb Q})\stackrel{\sim}{=} {\mathbb Q}[s_1,s_2],$$
of the fiber $N_{p_i}$.
By Nakajima's construction,
a weighted partition $\eta^i$ determines a $T$-equivariant class 
$$\CC_{\eta^i} \in A_T^*(\text{Hilb}(N_{p_i},d), \mathbb{Q})$$
in the
Chow ring of the Hilbert scheme of points.
In the theory of stable pairs, the weighted partition $\eta^i$
specifies relative
conditions via the boundary map
$$\epsilon_i: P_n(N/S,d)\rightarrow \text{Hilb}(N_{p_i},d).$$

An element $\eta\in {\mathcal P}(d)$ of the set of 
partitions of $d$ may be 
viewed as a 
weighted partition with all weights set to the identity class
$$1\in A^*_T(N_{p_i},{\mathbb Q})\ .$$
The Nakajima basis of $A_T^*(\text{Hilb}(N_{p_i},d), \mathbb{Q})$ consists of 
identity weighted partitions indexed by ${\mathcal P}(d)$.

Let $s_1,s_2 \in H^*_\T(\bullet)$ be the first Chern classes
of the standard representations of the first and second
$\C^*$-factors of $T$ respectively.
The $T$-equivariant intersection pairing in the Nakajima basis is
$$g_{\mu\nu}=\int_{\text{Hilb}(N_{p_i},d)} \CC_\mu \cup \CC_\nu =
\frac{1}{(s_1s_2)^{\ell(\mu)}} 
\frac{(-1)^{d-\ell(\mu)}} 
{{\mathfrak{z}}(\mu)}\ {\delta_{\mu,\nu}},$$
where
$${\mathfrak z}(\mu) =  \prod_{i=1}^{\ell(\mu)} \mu_i \cdot 
|\text{Aut}(\mu)|.$$
Let $g^{\mu\nu}$ be the inverse matrix.

%The notation $\eta([0])$ will be used to set all
%weights to $[0]\in A^*_T(N_{p_i},{\mathbb Q} )$.
%Since
%$$[0]= s_1s_2 \in A^*_T(N_{p_i}, {\mathbb Q} ),$$
%the weight choice has only a mild effect.

\subsection{Descendents}

We define descendents in the relative stable pairs
theory of local curves by the slant products
with universal sheaf following \cite{partone}.

There exists a universal sheaf
on the universal 3-fold  $\mathcal{N}$ over the moduli space $P_{n}(N/S,d)$,
$$\FF \rightarrow \mathcal{N}\ .$$
For a stable pair $[\OO\to F]\in P_{n}(N/S,d)$, the restriction of
$\FF$
to the fiber
 $$\mathcal{N}_{[\OO \to F]} \subset 
\mathcal{N}
$$
is canonically isomorphic to $F$.
Let
$$\pi_N\colon \mathcal{N} \to N,$$
$$\pi_P\colon \mathcal{N} \to P_{n}(N/S,d)$$
 be the canonical projections.

By the stability conditions for the relative theory of stable pairs,
 $\FF$ has a finite resolution 
by locally free sheaves.
Hence, the Chern character of the universal sheaf $\FF$ 
is well-defined.
By definition, the operation
$$
\pi_{P*}\big(\pi_N^*(\gamma)\cdot \text{ch}_{2+i}(\FF)
\cap(\pi_P^*(\ \cdot\ )\big)\colon 
H_*(P_{n}(N/S,d))\to H_*(P_{n}(N/S,d))
$$
is the action of the descendent $\tau_i(\gamma)$, where
$\gamma \in H^*(C,\Z)$.
The push-forwards are defined by $T$-equivariant
residues as in \cite{BryanP,lcdt}.

We will use bracket notation for descendents,
\begin{equation}\label{lwww}
\left\langle \prod_{j=1}^\ell \tau_{i_j}(\gamma_{j}) 
\right\rangle_{\!n,d}^{\! N,\eta^1,\ldots,\eta^r} =
\int _{[P_{n} (N/S,d)]^{vir}}
  \prod_{j=1}^\ell \tau_{i_j}(\gamma_{j})\ 
\prod_{i=1}^r \epsilon_i^*(\CC_{\eta^i})\ .
\end{equation}
The partition function is denoted by
$$
\ZZ_{d,\eta^1,\ldots,\eta^r}^{N/S}\left(   \prod_{j=1}^\ell \tau_{i_j}(\gamma_{j})
\right)^T
=\sum_{n}
\left\langle \prod_{j=1}^\ell \tau_{i_j}(\gamma_{j}) 
\right\rangle_{\!n,d}^{\! N,\eta^1,\ldots,\eta^r}q^n.
$$
The following basic result is proved in \cite{partone}.

\vspace{10pt}
\noindent{\bf Theorem.}
{\em $\ZZ_{d,\eta^1,\dots,\eta^r}
^{N/S}\big(   \prod_{j=1}^k \tau_{i_j}(\gamma_{j})
\big)^T$ is the 
Laurent expansion in $q$ of a rational function in $\mathbb{Q}(q,s_1,s_2)$.}
\vspace{10pt}

\subsection{Stationary theory}
Our main results here concern 
 stationary descendents in the stable pairs theory of 
local curves.
Let $$\mathsf{p}\in H^2(C,\mathbb{Z})$$ be the class of a point.
The {\em stationary descendents} are $\tau_k(\mathsf{p})$.
%$$\ZZ_{d,\eta^1,\dots,\eta^r}
%^{N/S}\big(   \prod_{j=1}^k \tau_{i_j}(\mathsf{p})
%\big)^T\ . $$
The methods of the paper,
while not fully applicable to other descendents, 
are much simpler and more effective
than the techniques of \cite{partone,part3}.

Our first result concerns reduction rules for stationary
descendents in the theory of local curves.
\begin{thm} \label{jjjttt2}
For $k>d$, there exist universal polynomials
$$f_{k,d}(x_1, \ldots,x_d) \in \mathbb{Q}(s_1,s_2)[
x_1, \ldots, x_d]$$
for which
the degree $d$ descendent theory of local curves satisfies
 the reduction rule
$$\tau_{k}(\mathsf{p}) \mapsto
f_{k,d}(\tau_1(\mathsf{p}), \ldots, \tau_d(\mathsf{p}))\ . $$
\end{thm}

Explicitly,  Theorem \ref{jjjttt2} yields the following equality for
${T}$-equivariant integrals:
\begin{multline*}
\ZZ^{N/S}_{d,\eta^1,\ldots,\eta^r}
\left(\tau_k(\mathsf{p}) \cdot \prod_{j=1}^{\ell}\tau_{i_j}(\gamma_j)
\right)^{{T}} = \\ 
\ZZ^{N/S}_{d,\eta^1,\ldots,\eta^r}
\left(
f_{k,d}(\tau_1(\mathsf{p}), \ldots \tau_d(\mathsf{p}))
 \cdot \prod_{j=1}^{\ell}\tau_{i_j}(\gamma_j)
\right)^{{T}}   \end{multline*}
for $k>d$ and all $\gamma_j \in H^*(C,\mathbb{Z})$.
Theorem \ref{jjjttt2} is proven in Section \ref{aaa}.

Via Theorem \ref{jjjttt2},
factorization properties of the relative conditions,
and the established rationality of the stable pairs
theory of local curves without insertions, we obtain our second result
in Section \ref{bbb}.

\begin{thm} \label{ghht4}
The stationary series
$\ZZ_{d,\eta^1,\dots,\eta^r}
^{N/S}\big(   \prod_{j=1}^k \tau_{i_j}(\mathsf{p})
\big)^T$ is the 
Laurent expansion in $q$ of a rational function 
$F(q,s_1,s_2)\in \mathbb{Q}(q,s_1,s_2)$
satisfying the functional equation
\[
F(q^{-1},s_1,s_2) = (-1)^{\Delta+|\eta|-\ell(\eta)
+
\sum_{j=1}^k
i_j}q^{-\Delta}F(q,s_1,s_2),
\]
where the constants are defined by 
$$\Delta = \int_{\beta}c_1(T_N),\ \ \ 
|\eta|=\sum_{i=1}^r |\eta^i|,\ \ \  \text{and} \ \ \
\ell(\eta)=\sum_{i=1}^r \ell(\eta^i) 
\ .$$
\end{thm} 

Here, $T_N$ is the tangent bundle of the 3-fold $N$, and
$\beta$ is the curve class given by $d$ times the $0$-section.
Our proof of Theorem \ref{ghht4} is much easier 
 than the
rationality results of \cite{partone}. Moreover,
 we do {not} know how to derive the
functional equation from the methods of \cite{partone}.
%However, the full results
%of \cite{partone} including descendents of the identity class
%and the odd cycles are deeper.

As a step in the proof of
Theorem \ref{ghht4},
we show the entire stationary descendent theory is
{\em determined} from the theory of local curves  without insertions
and the set of series
$$\mathsf{Z}^{\mathsf{cap}}_{d,(d)}( \tau_d(\mathsf{p}))^T =
\sum_{n}
\Big\langle \tau_{d}(\mathsf{p}) 
\Big\rangle_{\!n,d}^{\! N,(d)}q^n\ , \ \ \ \  d>0 \ .$$
Here, the cap geometry is $\Pp$ relative to $\infty\in \Pp$.
A central result of the paper is the following calculation.

\begin{thm}\label{ytytyt}
We have
$$\mathsf{Z}^{\mathsf{cap}}_{d,(d)}( \tau_d(\mathsf{p}))^T =
\frac{q^d}{d!}\left(\frac{s_1+s_2}{s_1s_2}\right)
\frac{1}{2}\sum_{i=1}^d  \frac{ 1+(-q)^{i}}{1-(-q)^i} \ . $$
\end{thm}

In the above formula, the coefficient of $q^d$, 
$$ \big\langle \tau_d, (d) \big\rangle_{\text{Hilb}(\com^2,d)}=
\frac{1}{2\cdot (d-1)!} \left(\frac{s_1+s_2}{s_1s_2}\right),$$
is the classical $T$-equivariant pairing on the
Hilbert scheme of $d$ points on $\C^2$.
The proof of Theorem \ref{ytytyt} 
%along with a combinatorial formula for $\mathsf{R}^d_{i}$ 
is given in Section \ref{ccc}.

Very few exact calculations for descendents in
 3-fold sheaf theories have previously been found. 
Theorem \ref{ytytyt} provides a closed
form for the most fundamental descendent series in the stationary
theory of local curves.
The derivation uses the localization methods of \cite{pt2} together
with an analysis of the fundamental solution of the quantum
differential equation of the Hilbert scheme of points of the
plane.

The descendent partition 
functions for the stable pairs theory of local
curves have very restricted
denominators when considered as rational functions in $q$
with coefficients in $\Q(s_1,s_2)$.
A basic result proven in Section 9 of
\cite{partone} is the following.

\vspace{10pt}
\noindent{\bf Theorem.}
{\em The denominators of the degree $d$ descendent partition functions
$\ZZ_{d,\eta^1,\dots,\eta^r}
^{N/S}\big(   \prod_{j=1}^k \tau_{i_j}(\mathsf{p})
\big)^T$
are  products of factors of the form $q^s$ and
$$1-(-q)^r$$
for $1\leq r \leq d$.}
\vspace{10pt}

Certainly the calculation of Theorem \ref{ytytyt}
is consistent with the denominator result.

\subsection{Acknowledgements}
We thank  J. Bryan, D. Maulik, A. Oblom\-kov, A. Okounkov, 
and
R. Thomas for several discussions about stable pairs, descendents,
and the quantum cohomology of the Hilbert scheme of points of
the plane. V. Shende's questions at the Newton Institute
about the $q\leftrightarrow q^{-1}$
symmetry for descendents
prompted us to work out the proof of the
functional equation.

R.P. was partially supported by NSF grant DMS-0500187
and DMS-1001154.
A.P. was supported by a NDSEG graduate fellowship.
The paper was completed while visiting the 
Instituto Superior T\'ecnico in Lisbon where
R.P. was supported by a Marie Curie fellowship and
a grant from the Gulbenkian foundation.

\section{Reduction for stationary descendents} \label{aaa}

\subsection{Cap geometry} \label{legger}
The capped 1-leg geometry concerns 
the trivial bundle,
$$N = \cO_{\PP^1} \oplus \cO_{\PP^1} \rightarrow \PP^1\ ,$$ 
relative to the fiber
$$N_\infty \subset N$$
over $\infty \in \PP^1$.
The total space $N$ naturally carries an action of a 
3-dimensional torus $$\mathbf{T} = T \times \com^*\ .$$
Here, $T$ acts as before by  scaling the
factors of $N$ and preserving the relative divisor $N_\infty$. 
The $\com^*$-action
on the base $\PP^1$ which fixes the points  $0, \infty\in \PP^1$ 
lifts to an additional $\com^*$-action on $N$ fixing
$N_\infty$.

The equivariant cohomology 
ring $H_{\mathbf{T}}^*(\bullet)$ is generated by
the Chern classes $s_1$, $s_2$, and $s_3$
of the standard representation of the three $\com^*$-factors.
At the $\mathbf{T}$-fixed point of $N$ over $0\in \PP^1$ the tangent
weights are specified as follows
\begin{enumerate}
\item[(i)] tangent weights of 
$-s_1$ and $-s_2$ along the fiber directions for the action of $T$,
\item[(ii)] tangent weight $-s_3$ along
$\PP^1$ for the action on $\com^*$.
\end{enumerate}
For the
$\mathbf{T}$-fixed point of $N$ over $\infty\in \PP^1$, the weights
are $-s_1,-s_2,s_3$.
We define 
\begin{equation}\label{pppw}
{\mathsf Z}^{\mathsf{cap}}_{d,\eta} 
\left(   \prod_{j=1}^\ell \tau_{i_j}(\gamma_{j})
\right)^{\mathbf{T}}
 =\sum _{n\in \Z }q^{n}
\int _{[P_{n} (N/N_\infty,d)]^{vir}}
  \prod_{j=1}^\ell \tau_{i_j}(\gamma_{j})\ 
\cup \epsilon_\infty^*(\mathsf{C}_{\eta}),
\end{equation}
by $\mathbf{T}$-equivariant residues
%{\footnote{The $T$-equivariant
%series associated to the cap will be denoted  
%$${\mathsf Z}^{\mathsf{cap}}_{d,\eta} 
%\left(   \prod_{j=1}^k \tau_{i_j}(\gamma_{j})
%\right)^T \ ,
%$$
%for $\gamma_j\in H^*(\Pp,\mathbb{Z})$.}
where $\gamma_j \in H^*_{\mathbf{T}}(\PP^1,\mathbb{Z})$.

\subsection{Reduction for the cap} 
\label{ddff}

Consider the following partition function for the cap  
\begin{equation}\label{keer}
\ZZ^{\mathsf{cap}}_{d,\eta}
\left(\tau_k([0]) \cdot \prod_{j=1}^{\ell}\tau_{i_j}(\gamma_j)
\right)^{\mathbf{T}}\ ,
\end{equation}
where $\gamma_j \in H_{\mathbf{T}}^*(\Pp,\mathbb{Z})$.

The $\mathbf{T}$-equivariant localization formula for \eqref{keer}
has two sides. The contribution over $0\in \Pp$ yields the 
descendent vertex $\mathsf{W}_\mu^{\mathsf{Vert}}$ of Section 2.6 of
\cite{partone}. We will follow here exactly the terminology of the
$\mathbf{T}$-fixed
point analysis of Sections 2.1-2.7 of \cite{partone}.
The contribution over $\infty\in \Pp$
yields rubber integrals discussed in Section 3.3
of \cite{partone}. 
While only the descendent vertex is required
for the proof of Theorem 1, the rubber theory
plays an essential role in the proof of Theorem 3. 

Let $Q_U$ determine a $\T$-fixed point
of the moduli space of stable pairs on the
affine chart associated to $0\in \PP^1$.
For each $x_1^ax_2^b \in \mu[x_1,x_2]$, let
$c_{a,b}$ be the {\em largest} integer
satisfying
$$x_1^ax_2^bx_3^{-c_{a,b}} \in Q_U\ .$$
The length of $Q_U$ is the sum of the $c_{a,b}$,
$$\ell(Q_U) = \sum_{(a,b)\in \mu} c_{a,b}\ .$$
The Laurent polynomial
\begin{equation}\label{kk999}
\FFF_{U} = \frac{1}{1-t_3}\sum_{(a,b)\in \mu} t_1^at_2^b t_3^{-c_{a,b}}
\end{equation}
plays a basic role.

In the formula in Section 2.6 of \cite{partone} for the descendent vertex
$\mathsf{W}^{\mathsf{Vert}}_\mu(\tau_k([0]))$, the descendent{\footnote{Here,
the class $[0]$ is the pull-back to $N$ of the fixed point $0\in \mathbb{P}^1$.}}
$\tau_k([0])$ enters via
\begin{multline*}
\frac{1}{s_1s_2}\text{ch}_{2+k}\big(\FFF_{U}\cdot 
(1-t_1)(1-t_2)(1-t_3)\big) = \\
\hspace{-50pt} \frac{1}{s_1s_2}\text{ch}_{2+k}\left( (1-t_1)(1-t_2)   
\sum_{(a,b)\in \mu} t_1^at_2^b t_3^{-c_{a,b}}\right) = \\
\frac{1}{s_1s_2}
\text{Coeff}_{z^{2+k}}\left( (1-e^{zs_1})(1-e^{zs_2})
\sum_{(a,b)\in \mu} e^{z(as_1+bs_2-c_{a,b}s_3)} \right)\ .
\end{multline*}
The third line exhibits the action of 
the descendent on $Q_U$ as
a symmetric function of the $d=|\mu|$ variables
\begin{equation} \label{gtyy}
\{ \ as_1+bs_2-c_{a,b}s_3 \ | \ (a,b)\in \mu \ \}
\end{equation}
with coefficients in $\Q[s_1,s_2]$.

In fact, the descendent
$\tau_k([0])$ is a symmetric function of degree $k$
in the variables \eqref{gtyy}.
The symmetric function is inhomogeneous with 
degree $k$ part equal to
$\frac{\mathfrak{p}_k}{k!}$
where $\mathfrak{p}_k$
is the power sum. Since the ring of symmetric
functions in $d$ variables is generated by
$\mathfrak{p}_1,\ldots, \mathfrak{p}_d$, 
we obtain
{\em universal reduction rules}.

Let $\mathfrak{t}_k$ be the symmetric function in $d$ variables
with coefficients in $\Q[s_1,s_2]$
defined by
$$\sum_{k=0}^\infty \mathfrak{t}_k z^{k+2}
= \frac{1}{s_1s_2}
(1-e^{zs_1})(1-e^{zs_2})
\sum_{n=0}^\infty {\mathfrak{p}_n} \frac{z^n}{n!}\ .$$
For $k>d$,
there are unique polynomials $f_{k,d}$  with coefficients
in $\Q(s_1,s_2)$ satisfying
\begin{equation}\label{u55}
\mathfrak{t}_k = f_{k,d}(\mathfrak{t}_1,
\ldots, \mathfrak{t}_d)\ .
\end{equation}
We have proven the following result.

\begin{prop} \label{jjtt}
In the degree $d$ theory of the 
$\mathbf{T}$-equivariant cap,
 the reduction rule
$$\tau_{k}([0]) \mapsto
f_{k,d}(\tau_1([0]), \ldots, \tau_d([0]))$$
holds universally when $k>d$.
\end{prop}

Explicitly, Proposition \ref{jjtt} yields the following equality for
$\mathbf{T}$-equivariant integrals:
\begin{multline*}
\ZZ^{\mathsf{cap}}_{d,\eta}
\left(\tau_k([0]) \cdot \prod_{j=1}^{\ell}\tau_{i_j}(\gamma_j)
\right)^{\mathbf{T}} = 
\ZZ^{\mathsf{cap}}_{d,\eta}
\left(
f_{k,d}(\tau_1(\mathsf{p}), \ldots, \tau_d(\mathsf{p}))
 \cdot \prod_{j=1}^{\ell}\tau_{i_j}(\gamma_j)
\right)^{\mathbf{T}}   \end{multline*}
for $k>d$.
Of course, Proposition \ref{jjtt} implies the
same result for the $T$-equivariant theory of the cap.

\subsection{Proof of Theorem \ref{jjjttt2}}
Consider the partition function for the relative
geometry $N/S$ over a curve $C$,  
$$\ZZ^{N/S}_{d,\eta^1, \ldots, \eta^r}
\left(\tau_k(\mathsf{p}) \cdot \prod_{j=1}^{\ell}\tau_{i_j}(\gamma_j)
\right)^{T}, \ \ \ \ \ \gamma_j \in H^*(C,\Z)\ .$$
Since the insertion $\tau_k(\mathsf{p})$ may be degenerated
to lie on a cap, Proposition \ref{jjtt} implies 
Theorem \ref{jjjttt2}. \qed

\subsection{Parity considerations}
We will need the following property of the reduction polynomials $f_{k,d}$ to obtain the functional equation of Theorem~\ref{ghht4}.

\begin{lem}\label{parity}
For every $k>d>0$, the reduction polynomial 
$$f_{k,d}\in\Q(s_1,s_2)[x_1,\ldots,x_d]$$
lies in the span of the monomials of the form
$x_1^{\sigma_1}\cdots x_d^{\sigma_d}$ where
$$\sum_{i=1}^d i\sigma_i \equiv k \mod 2\ .$$
%\[
%f_{k,d} = \sum_{|\sigma| \equiv k (\text{mod }2)}C_{\sigma}(s_1,s_2)T^{\sigma}.
%\]
\end{lem}
\begin{proof}
Using the homogeneity of $\mathfrak{t}_i$, we see
from \eqref{u55} 
the coefficient of $x_1^{\sigma_1}\cdots x_d^{\sigma_d}$
in $f_{k,d}$ is homogeneous as a rational function in $s_1$ and
$s_2$. Moreover the degree of the coefficient is congruent mod 2 to
$k-\sum_{i=1}^d i\sigma_i$. We need only show that these
degrees are all even.

 We write the descendent $\tau_k([0])$ as a symmetric function in the adjusted variables
\[
\{ \ as_1+bs_2-c_{a,b}s_3+\frac{s_1+s_2}{2} \ | \ (a,b)\in \mu \ \}.
\]
If we let $\mathfrak{p}'_k$ denote the $k$th power sum of these $d$ variables, then we have
\[
\sum_{k=0}^\infty \mathfrak{t}_k z^{k+2}
= \frac{1}{s_1s_2}
(e^{zs_1/2}-e^{-zs_1/2})(e^{zs_2/2}-e^{-zs_2/2})
\sum_{n=0}^\infty {\mathfrak{p}'_n} \frac{z^n}{n!}\ ,
\]
where $\mathfrak{t}_k$ is as in the proof of Proposition~\ref{jjtt}. Since
\[
(e^{zs_1/2}-e^{-zs_1/2})(e^{zs_2/2}-e^{-zs_2/2})
\]
is an even function of $s_1$ and $s_2$, 
the coefficients of the monomial of $f_{k,d}$ must have even degree.
\end{proof}

\section{Factorization and rationality} \label{bbb}

\subsection{Dependence upon the cap}
Consider the stationary series
\begin{equation}\label{jyq}
\ZZ_{d,\eta^1,\dots,\eta^r}
^{N/S}\left(   \prod_{j=1}^\ell \tau_{i_j}(\mathsf{p})
\right)^T \ .
\end{equation}
If $\ell=0$, then no descendents appear and the 
rationality  of the partition function \eqref{jyq}
has been proven in \cite{mpt,lcdt}. If $\ell>0$, 
each stationary descendent $\tau_i(\mathsf{p})$
can be degenerated to 
a distinct cap. Hence, the series \eqref{jyq}
is determined by:
\begin{enumerate}
\item[$\bullet$]
the stable pairs theory of local curves (without insertions), 
\item[$\bullet$] the 1-pointed caps
$\ZZ^{\mathsf{cap}}_{d,\eta}
\left(\tau_{k}(\mathsf{p})
\right)^T$. 
\end{enumerate}
In fact, we can do much better by using Theorem \ref{jjjttt2}.

\subsection{Factorization I} \label{r22}
If $k>d$, then we have
\begin{equation} \label{klw}
\ZZ^{\mathsf{cap}}_{d,\eta}
\Big(\tau_k(\mathsf{p})
\Big)^{{T}} = 
\ZZ^{\mathsf{cap}}_{d,\eta}
\Big(
f_{k,d}(\tau_1(\mathsf{p}), \ldots, \tau_d(\mathsf{p}))  \Big)^{{T}}   
\end{equation}
by Theorem \ref{jjjttt2}.
After expanding 
$f_{k,d}(\tau_1(\mathsf{p}), \ldots, \tau_d(\mathsf{p}))$
and degenerating each stationary descendent $\tau_i(\mathsf{p})$
to a distinct cap, we find the series \eqref{klw}
is determined by:
\begin{enumerate}
\item[$\bullet$]
the stable pairs theory of local curves (without insertions), 
\item[$\bullet$] the 1-pointed caps
$\ZZ^{\mathsf{cap}}_{d,\eta}
\left(\tau_{k\leq d}(\mathsf{p})
\right)^T$. 
\end{enumerate}

\subsection{Factorization II} \label{thth3}
We can further restrict the descendents $\tau_k(\mathsf{p})$
which occur on the caps by geometrically factoring the
parts of the relative condition $\eta$.

\begin{prop} \label{cprnnn}
The series 
$\ZZ^{\mathsf{cap}}_{d,\eta}
\left(\tau_{k\leq d}(\mathsf{p})
\right)^T$ are determined by
\begin{enumerate}
\item[$\bullet$]
the stable pairs theory of local curves (without insertions), 
\item[$\bullet$] the 1-pointed caps
$\ZZ^{\mathsf{cap}}_{c,(c)}
\left(\tau_{c}(\mathsf{p})
\right)^T$ for $1 \leq c \leq d$.
\end{enumerate}
\end{prop}

\begin{proof}
We proceed by induction on $d$.
If $d=1$, there is nothing to prove.
Assume Proposition \ref{cprnnn} holds for all degrees
less than $d$ and consider
$$\ZZ^{\mathsf{cap}}_{d,\eta}
\left(\tau_{k}(\mathsf{p})
\right)^T\ .$$
There are two main cases.

\vspace{10pt}
\noindent
Case $k<d$.
\vspace{10pt}

We consider the geometry of $\PP^2 \times \PP^1$
relative to the fiber 
$$\PP^2_\infty = \PP^2 \times \{ \infty \}\subset
\PP^2 \times \PP^1\ .$$
Let 
$\beta\in H_2(\PP^2\times \PP^1, \mathbb{Z})$
be the class of the section $\PP^1$ 
contracted over $\PP^2$.
The 2-dimensional torus $T$ acts on $\PP^2$
with fixed points $\xi_0,\xi_1,\xi_2\in \PP^2$.
The tangent weights can be chosen as follows:
$$-s_1,-s_2 \ \text{ for } \xi_0,\ \ 
s_1,s_1-s_2 \ \text{ for } \xi_1,\ \ 
s_2-s_1,s_2 \ \text{ for } \xi_2\ .$$
Let the partition $\eta$ have parts $\eta_1, \ldots, \eta_\ell$.
Let $\widetilde{\eta}$
be the cohomology weighted partition with
$\eta_1$ of weight $[\xi_0]\in H^*(\PP^2,\mathbb{Z})$ and
all of the other  parts assigned  weight 
$1\in H^*_T(\PP^2,\mathbb{Z})$.
The series
\begin{equation}\label{kedd}
\ZZ^{\PP^2\times\PP^1/\PP^2_\infty}_{d\beta,
\widetilde{\eta}}
\left(\tau_{k}([0])\right)^T  \in \Q[s_1,s_2][[q]]
\end{equation}
is well-defined.

The
virtual dimension of the moduli space
$P_n(\PP^2 \times \PP^1/\PP^2_\infty, d\beta)$
after the imposition of the boundary condition
$\widetilde{\eta}$ 
 is
$$2d - 2 - \sum_{i=1}^\ell (\eta_i-1) =
d+\ell -2 \geq d-1\ . $$
The dimension of the integrand $\tau_k([0])$
is $k< d$.
Hence, the integrals
$$\left\langle \tau_{k}([0]) 
\right\rangle_{\!n,d\beta}^{\! \PP^2\times \PP^1,\widetilde{\eta}}
= \int _{[P_{n} (\PP^2\times \PP^1/\PP^2_\infty,d\beta)]^{vir}}
  \tau_{k}([0]) \cup
\epsilon^*(\CC_{\widetilde{\eta}})\ $$
arising as coefficients of  \eqref{kedd}
have degree at most $0$ in $\mathbb{Q}[s_1,s_2]$.
If  the degree is negative, 
then the series \eqref{kedd}  vanishes.

The degree of \eqref{kedd} is 0 only when
$k=d-1$ and $\eta=(d)$.
The moduli space then lies entirely
in $$\com^2 \times \PP^1 \subset
\PP^2 \times \PP^1\ $$
 where $\com^2\subset \PP^2$
is the $T$-invariant affine containing
$\xi_0$ (corresponding to the cohomology  weight $[\xi_0]$
on the part $d$).
By the basic divisibility results of \cite{mpt,lcdt},
the linear factor $s_1+s_2$ must divide the $q^n$
coefficient of \eqref{kedd} for all
$n>d$. Since the invariant
is of degree 0, the divisibility by $s_1+s_2$
is impossible unless all such
coefficients vanish.
Since the leading term of \eqref{kedd} is $q^d$,
we conclude \eqref{kedd} is a monomial in $q$.

If $k<d$, we have calculated the series \eqref{kedd}.
Direct
calculation of \eqref{kedd} by $T$-equivariant
localization yields
a single term equal to
\begin{equation}\label{medd}
\ZZ^{\mathsf{cap}}_{d,\eta}
\left(\tau_{k}(\mathsf{p})
\right)^T
\end{equation}
up to an $s_1s_2$ factor.
The $T$-equivariant localization formula
for the relative geometry $\PP^2\times\PP^1/\PP^2_\infty$
in the class $d\beta$ distributes the parts of
$\widetilde{\eta}$ among the $T$-fixed points 
$$\xi_0,\xi_1,\xi_2 \in \PP^2 \ .$$
The term equal to \eqref{medd} arises when all parts are
distributed to $\xi_0$.
Since the first part of $\widetilde{\eta}$ must be
distributed to $\xi_0$,
the remaining terms are known by the induction
hypothesis. 
Hence, we have calculated \eqref{medd}.

\vspace{10pt}
\noindent
{\em Case} $\ell>1$.
\vspace{10pt}

The dimension estimates as above show
the series
\begin{equation}\label{keddd}
\ZZ^{\PP^2\times\PP^1/\PP^2_\infty}_{d\beta,
\widetilde{\eta}}
\left(\tau_{k}([0])\right)^T  \in \Q[s_1,s_2][[q]]
\end{equation}
is degree at most 0 in $s_1$ and $s_2$.
The series \eqref{keddd} must vanish in the negative
degree case.

The degree of \eqref{keddd} is 0 only when
$k=d$ and $\eta=(d_1,d_2)$.
In the degree 0 case, the invariant
\eqref{keddd} is independent of $s_1$ and
$s_2$, so we may calculate \eqref{keddd}
in the specialization $s_1+s_2=0$.
In the $T$-equivariant localization
of \eqref{keddd}, the
terms at $\xi_0$ all have vanishing
 coefficients
of $q^{n>d}$ 
in the specialization $s_1+s_2=0$.
The terms away from $\xi_0$ are
known inductively. Hence, \eqref{keddd}
is determined.

If $\ell>1$, we have calculated the series \eqref{keddd}.
As before, 
the $T$-equivariant localization
formula for  \eqref{keddd}
 yields
a single term equal to
$\ZZ^{\mathsf{cap}}_{d,\eta}
\left(\tau_{k}(\mathsf{p})
\right)^T$
up to an $s_1s_2$ factor.
The remaining terms are known by the induction
hypothesis. We have calculated 
$\ZZ^{\mathsf{cap}}_{d,\eta}
\left(\tau_{k}(\mathsf{p})
\right)^T$.

\vspace{+10pt}
The only possibility not covered by the two above cases
is 
 the 1-pointed cap
\begin{equation}\label{fgh}
\ZZ^{\mathsf{cap}}_{d,{(d)}}
\left(\tau_{d}(\mathsf{p})
\right)^T \in \Q[s_1,s_2][[q]] \ .
\end{equation}
The factorization methods do not inductively
determine \eqref{fgh}. 
\end{proof}

\subsection{Proof of Theorem \ref{ghht4}}
The methods of Sections \ref{r22}-\ref{thth3} 
provide an effective algorithm
for calculating an arbitrary degree $d$ 
stationary series \eqref{jyq} in terms of
\begin{enumerate}
\item[$\bullet$]
the stable pairs theory of local curves (without insertions), 
\item[$\bullet$] the 1-pointed caps
$\ZZ^{\mathsf{cap}}_{c,(c)}
\left(\tau_{c}(\mathsf{p})
\right)^T$ for $1 \leq c \leq d$.
\end{enumerate}
The partition functions of the stable pairs theory
of local curves (without insertions) are rational and
satisfy the functional equation of Theorem~\ref{ghht4},
see Theorems 2 and 3 of \cite{lcdt}.
The steps in the effective algorithm preserve the functional
equation. 
For the Factorization I step, Lemma~\ref{parity} is 
needed to ensure that the total weight of the descendent 
insertions does not change parity.
Theorem \ref{ghht4} then follows
from Theorem \ref{ytytyt} proven in Section \ref{ccc} below 
together with the observation that the rational functions 
appearing there satisfy the functional equation.

\qed

\section{Localization formalism}
\label{locrev}
\subsection{Formula}
The $\mathbf{T}$-equivariant localization formula  
for the capped 1-leg descendent
vertex is the following:
\begin{equation}\label{fred}
{\mathsf Z}^{\mathsf{cap}}_{d,\eta} 
\left(  \prod_{i=1}^k\tau_{i_j}([0])
\right)^{\mathbf{T}} = \sum_{|\mu|=d}
\bW_\mu^{\mathsf{Vert}} \left(\prod_{j=1}^k\tau_{i_j}([0])      \right) \cdot
{\bW_\mu^{(0,0)}} \cdot \mathsf{S}^{\mu}_{\eta}\ .
\end{equation}
The result is a consequence of \cite{GraberP} applied
to stable pairs theory of the cap \cite{pt2} ---
see Section 3.4 of \cite{partone}.
The form is the same as the Donaldson-Thomas localization formulas
used in \cite{moop,lcdt}.

The right side of 
localization formula is expressed in term of
three parts of different geometric origins:
\begin{enumerate}
\item[$\bullet$] the vertex term  $\bW_\mu^{\mathsf{Vert}} \left(
\prod_{j=1}^k
\tau_{i_j}([0])      \right)$   over $0\in \PP^1$,
\item[$\bullet$] the edge term   $\bW_\mu^{(0,0)}$,
\item[$\bullet$] the rubber integrals 
$\mathsf{S}^{\mu}_{\eta}$
over $\infty \in \PP^1$.
\end{enumerate}
The vertex term has been explained 
(for $i=1$)
already in Section \ref{ddff}. The edge term
$\bW^{(0,0)}_\mu$ is simply the inverse product
of the tangent weights of the Hilbert scheme of points of $\C^2$
at the $T$-fixed point corresponding to the partition $\mu$.
We review the rubber integrals here.

\subsection{Rubber theory}
The stable pairs theory of {\em rubber}{\footnote{We
follow the terminology and conventions of the
parallel rubber discussion for the local Donaldson-Thomas
theory of curves treated in \cite{lcdt}.}} naturally arises at the
boundary of $P_n(N/N_\infty,d)$.
Let $R$ be a rank 2 bundle of level $(0,0)$ over $\Pp$. Let 
 $$R_0, R_\infty\subset R$$ 
denote the fibers over $0, \infty\in \Pp$.
The 1-dimensional torus $\C^*$ acts on $R$ via the symmetries of
$\Pp$. 
Let $P_n(R/R_0\cup R_\infty,d)$ be the relative moduli space
of ideal sheaves, and let
$$P_n(R/R_0 \cup R_\infty,d)^\circ \subset P_n(R/R_0\cup R_\infty,d)$$
denote the open set with finite stabilizers for the $\C^*$-action
and {\em no} destabilization over $\infty\in \Pp$.
The rubber moduli space,
$${P_n(R/R_0\cup R_\infty,d)}^\sim  
= P_n(R/R_0 \cup R_\infty,d)^\circ/\C^*,$$
denoted by a superscripted tilde,
is determined by the (stack) quotient. The moduli space is 
empty unless $n>d$.
The rubber theory of $R$ is defined by integration against the
rubber virtual class,
 $$[{P_n(R/R_0\cup R_\infty,d)}^\sim ]^{vir}.$$ 
All of the above rubber constructions are $T$-equivariant for the
scaling action on the fibers of $R$ with weights $s_1$ and $s_2$.

The rubber moduli space $P_n(R/R_0\cup R_\infty, d)^\sim$ carries
a cotangent line at the dynamical point $0 \in \Pp$. Let
$$\psi_0 \in A^1_T({P_n(R/R_0\cup R_\infty,d)}^\sim, {\mathbb Q})$$
denote the associated cotangent line class.
Let $$\mathsf{P}_\mu \in A^{2d}_T(\text{Hilb}(\C^2,d),\mathbb{Z})$$
be the class corresponding to the $T$-fixed point determined
by the monomial ideal $\mu[x_1,x_2]\subset \C[x_1,x_2]$.

In the localization formula for the cap, special
rubber integrals with relative conditions $\mathsf{P}_\mu$ over $0$ and $\CC_\eta$
(in the Nakajima basis) over $\infty$ arise. Let
\begin{equation*} 
\mathsf{S}^\mu_\eta =   
%e(\text{Tan}_\mu)\cdot  
\sum_{n\geq d} q^{n}
\left\langle \mathsf{P}_\mu \ \left| \ \frac{1}{s_3-\psi_0}  \ \right|\ \CC_\eta 
\right\rangle_{n,d}^{
\sim}\ \in \Q(s_1,s_2,s_3)((q)) \ .
\end{equation*}
%where $e(\text{Tan}_\mu)$ is the $T$-equivariant Euler
%class of the tangent space of $\mu[x_1,x_2]$ in 
%$\text{Hilb}(\C^2,d)$. 
The bracket on the right is the rubber
integral defined by $T$-equivariant
residues. If $n=d$, the rubber moduli space in undefined ---
the bracket is then taken to be the $T$-equivariant intersection pairing
between the classes $\mathsf{P}_\mu$ and $\CC_\eta$ in 
$\text{Hilb}(\C^2,d)$.

\section{Calculation of $\mathsf{Z}^{\mathsf{cap}}_{d,(d)}( \tau_d(\mathsf{p}))^{\T}$}
\label{ccc}

\subsection{Dimension}

The notation $(d[0])$ will be used to assign the 
weight $[0]\in A^*_T(\C^2,{\mathbb Q} )$ to the part $d$.
Since
$$[0]= s_1s_2 \in A^*_T(\C^2, {\mathbb Q} ),$$
we see{\footnote{We will consider descendents here
equivariant with respect to the 3-torus $\T$ of Section \ref{aaa}.}}
$$\ZZ^{\mathsf{cap}}_{d,(d)}
\left(\tau_{d}(\mathsf{p})
\right)^\T = \left( \frac{1}{s_1s_2} \right)
\ZZ^{\mathsf{cap}}_{d,(d[0])}
\left(\tau_{d}(\mathsf{p})
\right)^\T$$
After imposing the boundary condition
$(d[0])$, the moduli space 
$$P_n(\PP^2\times \PP^1/\PP^2_\infty,d\beta)$$
is compact of virtual dimension $d-1$.

The moduli space $P_n(\PP^2\times \PP^1/\PP^2_\infty,d\beta)$
is empty for $n<d$ and isomorphic to $\text{Hilb}(\C^2,d)$
for $n=d$. Hence, the leading term of the series
$\ZZ^{\mathsf{cap}}_{d,(d[0])}
\left(\tau_{d}(\mathsf{p})
\right)^\T$
is the classical pairing
\begin{equation}\label{ktt}
q^d \Big\langle \tau_d, \mathsf{C}_{(d[0])} \Big \rangle_{\text{Hilb}(\com^2,d)} \in
\mathbb{Q}[s_1,s_2]\ . 
\end{equation}

The class $\tau_d$ is defined as follows.
Let
${\mathbb{F}}_0$ be the universal quotient sheaf on 
$\text{Hilb}(\C^2,d)\times \C^2$. Then,
\begin{equation}\label{mrr}
\tau_d=\pi_*\Big( {\text{ch}}_{2+d}({\mathbb{F}}_0)\Big) 
\in A^d_T({\text{Hilb}}(\C^2,d))
\end{equation}
where $\pi$ is the projection
$$\pi: \text{Hilb}(\C^2,d)\times \C^2 \rightarrow
\text{Hilb}(\C^2,d)\ .
$$

\begin{lem} \label{pwpw}
$
\ZZ^{\mathsf{cap}}_{d,{(d[0])}}
\left(\tau_{d}(\mathsf{p})
\right)^\T= (s_1+s_2)\cdot q^d F(d)(q)\ 
$
for   $F(d) \in \Q[[q]]$.
\end{lem}

\begin{proof}
By compactness of the underlying moduli spaces of
pairs, we see the series
$
\ZZ^{\mathsf{cap}}_{d,{(d[0])}}
\left(\tau_{d}(\mathsf{p})
\right)^\T$ must lie in 
$\mathbb{Q}[s_1,s_2,s_3][[q]]
$.
The leading $q^d$ coefficient certainly has no
$s_3$ dependence by \eqref{ktt}. By dimension considerations, the
leading $q^d$ coefficient must be linear 
and thus, by symmetry, a multiple of $s_1+s_2$.
For the coefficient of $q^{n>d}$, divisibility by $s_1+s_2$
is obtained from \cite{mpt,lcdt}.
\end{proof}

\subsection{Localization}
We wish to compute the series
\[
F(d) = \frac{s_1s_2}{s_1+s_2}q^{-d}\ \mathsf{Z}^{\mathsf{cap}}_{d,(d)}(\tau_{d}([0]))^{\T} \in \Q[[q]]\subset\Q(s_1,s_2,s_3)[[q]]
\]
introduced in Lemma~\ref{pwpw}. Via the  
 localization formula \eqref{fred}, we have
\[
F(d) = \frac{s_1s_2}{s_1+s_2}q^{-d}\sum_{|\mu|=d}
\bW_\mu^{\mathsf{Vert}} (\tau_{d}([0])) \cdot  
\bW_\mu^{(0,0)} \cdot \mathsf{S}^{\mu}_{(d)}\ .
\]
We will separate the classical terms occuring on the right side.

By definition, the classical term of $\bW_\mu^{\mathsf{Vert}} (\tau_{d}([0])$
is the leading $q^d$ term. Let
$$\FFF_{\mu} = \sum_{(a,b)\in \mu} t_1^at_2^b.$$
We write the vertex as
$$\bW_\mu^{\mathsf{Vert}} (\tau_{d}([0])) = 
\frac{q^d}{s_1s_2}\ \ch_{d+2}\big(\FFF_\mu\cdot(1-t_1)(1-t_2)\big) 
+ \widehat{\bW}_\mu^{\mathsf{Vert}} (\tau_{d}([0]))\ $$
where $\widehat{\bW}_\mu^{\mathsf{Vert}} (\tau_{d}([0]))$
represents all the higher order terms in $q$.
Similarly, we write
$$\mathsf{S}^{\mu}_{(d)} = \langle \JJ_\mu,\CC_{(d)}\rangle + 
\widehat{\mathsf{S}}^{\mu}_{(d)} $$
where the leading term
 $\langle \JJ_\mu,\CC_{(d)}\rangle$
is the $T$-equivariant pairing on $\text{Hilb}(\mathbb{C}^2,d)$.

Using the above formulas with the leading classical terms, 
we rewrite the result of the localization formula as
\begin{eqnarray*}
F(d) &=& \ \ \ \sum_{|\mu|=d}\frac{1}{s_1+s_2}\ch_{d+2}(\FFF_\mu\cdot(1-t_1)(1-t_2))\cdot \bW_\mu^{(0,0)}\cdot \langle \JJ_\mu,\CC_{(d)}\rangle \\ \nonumber
&& + \sum_{|\mu|=d}\frac{s_1s_2}{s_1+s_2}q^{-d}\ \widehat{\bW}_\mu^{\mathsf{Vert}}
(\tau_{d}([0]))\cdot \bW_\mu^{(0,0)}\cdot
\langle \JJ_\mu,\CC_{(d)}\rangle \\ \nonumber
&&+ \sum_{|\mu|=d}\frac{1}{s_1+s_2}\ch_{d+2}(\FFF_\mu\cdot(1-t_1)(1-t_2))\cdot
\bW_\mu^{(0,0)}\cdot \widehat{\mathsf{S}}^{\mu}_{(d)} \\ \nonumber
&&+ \sum_{|\mu|=d}\frac{s_1s_2}{s_1+s_2}q^{-d}\
\widehat{\bW}_\mu^{\mathsf{Vert}} (\tau_{d}([0]))\cdot 
\bW_\mu^{(0,0)}\cdot
\widehat{\mathsf{S}}^{\mu}_{(d)}.
\end{eqnarray*}
The first line on the right is the classical pairing 
$$F_0(d) = \frac{1}{s_1+s_2}
\Big\langle \tau_{d}, \CC_{(d[0])}\Big\rangle\in\Q$$
which we will compute in Proposition \ref{classicald} below. 
We will compute the difference
 $$\widehat{F}(d) = F(d)-F_0(d)$$
 by evaluating each of the other three terms at $s_2=-s_1$, 
expanding as a Laurent series in $\frac{s_3}{s_1}$,
 and taking the constant term.

Both $\widehat{\bW}_\mu^{\mathsf{Vert}}$ and $\widehat{\mathsf{S}}^{\mu}_{(d)}$ 
are divisible by $s_1+s_2$. Therefore,
 the fourth term in  the formula for $F(d)$ vanishes after
 the substitution $s_2=-s_1$.
Only two terms,
\begin{multline*}
\widehat{F}(d) = \sum_{|\mu|=d}\left(\frac{s_1s_2}{s_1+s_2}q^{-d}\
\widehat{\bW}_\mu^{\mathsf{Vert}} (\tau_{d}([0]))\cdot \bW_\mu^{(0,0)}\cdot\langle \JJ_\mu,\CC_{(d)}\rangle\right)\Big|_{s_2=-s_1} \\
+ \sum_{|\mu|=d}\left(
\frac{1}{s_1+s_2}
\ch_{d+2}(\FFF_\mu\cdot(1-t_1)(1-t_2))\cdot \bW_\mu^{(0,0)}\cdot 
\widehat{\mathsf{S}}^{\mu}_{(d)}\right)\Big|_{s_2=-s_1},
\end{multline*}
remain.

We evaluate the two above
terms separately. The first requires detailed knowledge 
of the vertex factor $$\frac{s_1s_2}{s_1+s_2}q^{-d}\
\widehat{\bW}_\mu^{\mathsf{Vert}}(\tau_{d}([0]))\Big|_{s_2=-s_1}$$
and is evaluated in Section \ref{vcv}.
The second requires detailed knowledge of the rubber 
factor $$\frac{1}{s_1+s_2}\widehat{\mathsf{S}}^{\mu}_{(d)}\Big|_{s_2=-s_1}$$
and is evaluated in Section \ref{rcr}.

\subsection{Vertex calculation}\label{vcv}

We begin with the first term
\begin{equation}\label{gttgt}
\sum_{|\mu|=d}\left(\frac{s_1s_2}{s_1+s_2}q^{-d}\
\widehat{\bW}_\mu^{\mathsf{Vert}} (\tau_{d}([0]))\cdot \bW_\mu^{(0,0)}\cdot\langle \JJ_\mu,\CC_{(d)}\rangle\right)\Big|_{s_2=-s_1}
\end{equation}
of  $\widehat{F}(d)$.
The pairing $\langle \JJ_\mu,\CC_{(d)}\rangle$ has a simple expression mod $s_1+s_2$,
\begin{equation}\label{jc}
\Big\langle \JJ_\mu,\CC_{(d)}\Big\rangle\Big|_{s_2=-s_1} = 
\frac{(-1)^{d-1}(d-1)!}{ \dim\mu }\chi^\mu\big((d)\big)s_1^{d-1}.
\end{equation}
Here, $\dim\mu$ is the dimension of the irreducible representation of 
the symmetric group $\Sigma_d$ corresponding to the partition $\mu$,
 and $\chi^\mu$ is the associated character.
The proof of \eqref{jc} is obtained directly from 
the Jack polynomial expression for
the $T$-fixed points of $\text{Hilb}(\com^2,d)$, see Section 3.7 of
\cite{hilb1}.{\footnote{Our variable conventions here differ slightly from
\cite{hilb1}. Specifically, our $s_i$ correspond to $-t_i$ in \cite{hilb1}.}}

The character $\chi^\mu$  vanishes on a $d$-cycle unless $\mu$
is of the following simple form
$$\alpha_a = (a+1,1,\ldots,1)$$ 
for  $0\le a \le d-1$. We have
$$\chi^{\alpha_a}\big((d)\big)= (-1)^{d-1-a}\ .$$
We will restrict to the case 
$\mu=\alpha_a$ and replace the sum over $\mu$ with a sum over $a$. 
The dimension formula
$$\dim \alpha_a = \binom{d-1}{a}$$
holds. The constant
$$b=d-1-a$$
will occur often below.

The edge factor $\bW_\mu^{(0,0)}$ is also easy to compute after the
 evaluation  $s_2=-s_1$:
\begin{equation}\label{edge}
\bW_\mu^{(0,0)}\Big|_{s_2=-s_1} = \frac{(-1)^d(\dim\mu)^2}{(d!)^2}s_1^{-2d}.
\end{equation}
The dimension of $\mu$ here enters via the hook length formula.

The most complicated part of the calculation
is the vertex factor $\widehat{\bW}_\mu^{\mathsf{Vert}} 
\left(\tau_{d}([0])\right)$ for $\mu = \alpha_a$.
From Section 2.6 of \cite{partone},
\begin{multline*}
\widehat{\bW}_\mu^{\mathsf{Vert}} \left(\tau_{d}([0])\right) 
= \\ \sum_{Q_U:\ \ l(Q_U)>0}
\frac{q^{d+l(Q_U)}}{s_1s_2}\ch_{d+2}
\big(\bF_U\cdot(1-t_1)(1-t_2)(1-t_3)\big)\cdot e(-\bV_U),
\end{multline*}
where the sum runs over $\T$-fixed loci $Q_U$ of positive length. 
Recall, the 
the $\T$-fixed loci correspond to box configurations defined by height functions $c_{a,b}$ on the partition $\mu$ determining 
$\bF_U$ by formula \eqref{kk999}.
The term $\bV_U$ is expressed in terms of $\bF_U$ in Section 2.5 of 
\cite{partone}.

In the case $\mu = \alpha_a$, a straightforward
calculation shows the vertex weight $e(-\bV_U)$ is divisible by 
$(s_1+s_2)^2$ unless the box configuration is a cylinder (of height $h>0$) 
under a rim hook $\eta$ of $\mu$.\footnote{The
divisibility statement is actually true for any $\mu$.} 
We break the sum into terms by the size $r$ of $\eta$. 
When $r = d$, the only possibility for the rim hook is $\eta=\mu$. 
The corresponding vertex weight is 
\[
\frac{1}{s_1+s_2}e(-\bV_U)\Big|_{s_2=-s_1} 
= \frac{(-1)^{dh+1}}{hs_3}\left(1+\frac{hs_3}{s_1}\sum_{\substack{i=-b \\ i\ne 0}}^a \frac{1}{i} + 
%\OO\left(\left(\frac{s_3}{s_1}\right)^2\right)
\ldots \right).
\]
Here and below, the dots on the right stand for terms of order 2 and
higher in $\frac{s_3}{s_1}$.
For each $r < d$, there are at most two such rim hooks, 
depending on whether $a \ge r$ and whether $d-1-a \ge r$. 
For $a\geq r$, we find
\begin{multline*}
\frac{1}{s_1+s_2}e(-\bV_U)\Big|_{s_2=-s_1} = \\ \ \ \ \
\frac{(-1)^{rh+1}}{hs_3}\left(1+\frac{hs_3}{s_1}\left(\frac{1}{d}-\frac{1}{d-r}-\frac{1}{r}+\sum_{i=a-r+1}^a \frac{1}{i}\right)
 + 
%\OO\left(\left(\frac{s_3}{s_1}\right)^2\right)
\ldots
\right).
\end{multline*}
For $d-1-a \ge r$, the answer is obtained by symmetry by
interchanging $s_1$ and $s_2$.
The symmetry propagates through the entire calculation of \eqref{gttgt}.
After setting $s_2=-s_1$,
we will take the constant term of the Laurent
expansion in 
$$\frac{s_3}{s_1} = -\frac{s_3}{s_2}\ . $$
Hence, we can treat the symmetry as exact.

After putting all the terms together  and 
inserting the descendent factors, we obtain 
for
$$\frac{s_1s_2}{s_1+s_2}q^{-d}\bW_\mu^{\mathsf{Vert}+} (\tau_{d}([0]))
\Big|_{s_2=-s_1}$$
the following formula:
\footnotesize
\begin{multline*} \ \ \ \ \ \ \ \ 
\sum_{h=1}^{\infty}\frac{(-1)^{dh+1}q^{dh}}{hs_3}\left(1+\frac{hs_3}{s_1}
\sum_{\substack{i=-b \\ i\ne 0}}^a \frac{1}{i} + 
%\OO\left(\left(\frac{s_3}{s_1}\right)^2\right)
\ldots
\right) \\
  \cdot
\ch_{d+2}\left(-t_1^{a+1}t_3^{-h}+t_1^at_3^{-h}+t_1^{-b}t_3^{-h}-t_1^{-b-1}t_3^{-h}\right) \\{}
\\
+2\sum_{r=1}^{a}\sum_{h=1}^{\infty}\frac{(-1)^{rh+1}q^{rh}}{hs_3}\left(1+\frac{hs_3}{s_1}\left(\frac{1}{d}-\frac{1}{d-r}-\frac{1}{r}+\sum_{i=a-r+1}^a \frac{1}{i}\right) + 
%\OO\left(\left(\frac{s_3}{s_1}\right)^2\right)
\ldots \right) \\
\cdot \ch_{d+2}\left(-t_1^{a+1}t_3^{-h}+t_1^at_3^{-h}+t_1^{a-r+1}t_3^{-h}-t_1^{a-r}t_3^{-h}\right)
\end{multline*}
\normalsize
Here, we have included the symmetry discussed above.

After combining all of the parts
of \eqref{gttgt}, summing over $\mu = \alpha_a$, and taking 
the constant term when expanded as a Laurent series in $\frac{s_3}{s_1}$, we obtain an expression of the form
\[
\sum_{r=1}^d A_r\frac{(-q)^r}{1-(-q)^r}\ .
\]
The explicit formulas for $A_r$ depend upon two cases.
For $A_d$, we have
\footnotesize
\begin{multline*} 
%\label{casead}
{(-1)^{d-1}}{d\cdot d!(d+2)!}
\ A_d=\\
\sum_{a+b=d-1}(-1)^a\binom{d-1}{a}
\cdot
\sum_{\substack{i=-b \\ i\ne 0}}^a\frac{1}{i}\left(-(-b-1)^{d+2}+(-b)^{d+2}+a^{d+2}-(a+1)^{d+2}\right) \\
- {(d+2)}\sum_{a+b=d-1}(-1)^a\binom{d-1}{a}\left(-(-b-1)^{d+1}+(-b)^{d+1}+a^{d+1}-(a+1)^{d+1}\right),
\end{multline*}
\normalsize
and for $r<d$, we have
\footnotesize
\begin{multline*} 
%\label{caseanotd}
\frac{(-1)^{d-1}}{2} d\cdot d!(d+2)!\
A_r = \\ 
\sum_{\substack{a+b=d-1 \\ a\ge r}}(-1)^a\binom{d-1}{a}\left(\frac{1}{d}-\frac{1}{d-r}-\frac{1}{r}+\sum_{i=a-r+1}^a\frac{1}{i}\right) \\
\ \ \ \ \ \ \ \ \ \ \ \ \ \ \ \ \ \ \ \ \ \ \  \ \ \ \ \ \ \ \ \ \ \ \ \ \ 
\cdot\left(-(-b-1)^{d+2}+(-b)^{d+2}+a^{d+2}-(a+1)^{d+2}\right) \\
- (d+2)
\sum_{a=r}^{d-1}(-1)^a\binom{d-1}{a}\left(-(a-r)^{d+1}+(a-r+1)^{d+1}+a^{d+1}-(a+1)^{d+1}\right) .
\end{multline*}
\normalsize
While the above formulas for $A_d$ and $A_{r<d}$ look unpleasantly complicated,
a remarkable cancellation will occur in Section \ref{hhttyy}.

\subsection{Rubber calculation} \label{rcr}

Evaluating the second term 
\begin{equation}\label{hhthh}
\sum_{|\mu|=d}\left(
\frac{1}{s_1+s_2}
\ch_{d+2}(\FFF_\mu\cdot(1-t_1)(1-t_2))\cdot \bW_\mu^{(0,0)}\cdot 
\widehat{\mathsf{S}}^{\mu}_{(d)}\right)\Big|_{s_2=-s_1}
\end{equation}
 requires care in moving between two bases for the equivariant 
cohomology of the Hilbert scheme of $\C^2$ (which we identify with 
Fock space): the Nakajima basis $\{\CC_\lambda\}$ and the $T$-fixed point basis $\{\JJ_\lambda\}$. 
The change of basis formula is simple mod $s_1+s_2$:
\begin{equation} \label{changebasis} %doublecheck signs here
\JJ_\lambda = \sum_\mu\frac{(-1)^{\ell(\mu)}d!}{\dim\lambda}\chi^\lambda(\mu)s_1^{d+\ell(\mu)}\CC_\mu.
\end{equation}
Our $\widehat{\mathsf{S}}^{\mu}_{(d)}$ should be viewed as having upper index given in the $T$-fixed point basis but lower index in the Nakajima basis.

The main tool for evaluating $\mathsf{S}$
is the quantum differential equation of \cite{lcdt} valid
also for stable pairs \cite{mpt},
\begin{equation} \label{xxqq}
s_3 q\frac{d}{dq} {\mathsf S}
 = \MM {\mathsf S} - {\mathsf S} \MM(0) \ .
\end{equation}
Here, $\mathsf{S}$ has both components indexed by the Nakajima 
basis and is viewed as an operator on Fock space, see \cite{lcdt}.
The operator
$\MM$ is defined on Fock space by{\footnote{The operator $\MM$ was 
found earlier in the
quantum cohomology of the Hilbert scheme of points of $\com^2$ \cite{hilb1}.
A parallel occurance appears in the local Gromov-Witten theory
of curves \cite{BryanP}.}}
\begin{multline}
  \label{theM} 
\MM(q,s_1,s_2) = (s_1+s_2) \sum_{k>0} \frac{k}{2} \frac{(-q)^k+1}{(-q)^k-1} \,
 \alpha_{-k} \, \alpha_k  + \\
\frac12 \sum_{k,l>0} 
\Big[s_1 s_2 \, \alpha_{k+l} \, \alpha_{-k} \, \alpha_{-l} -
 \alpha_{-k-l}\,  \alpha_{k} \, \alpha_{l} \Big] \,.
\end{multline}
The $q$-dependence of $\MM$ is only in the first sum in \eqref{theM}.
The operator $\MM(0)$ is the $q^0$-coefficient of $\MM$.

From the differential equation \eqref{xxqq}, we find
\begin{equation*} 
s_3 q\frac{d}{dq} {\mathsf S}
 =  (s_1+s_2)\mathsf{A} + [\mathsf{B}, {\mathsf S}]  \ \ \ \mod (s_1+s_2)^2\ .
\end{equation*}
The first term is
$$\mathsf{A} = \left(\sum_{k>0} \frac{k}{2} \frac{(-q)^k+1}{(-q)^k-1} \,
 \alpha_{-k} \, \alpha_k \right) \circ \mathsf{S}(0) + 
\mathsf{S}(0) \circ
 \left(\sum_{k>0} \frac{k}{2}  \,
 \alpha_{-k} \, \alpha_k \right). $$
The operator in the second term is
$$\mathsf{B} = 
\frac12 \sum_{k,l>0} 
\Big[s_1 s_2 \, \alpha_{k+l} \, \alpha_{-k} \, \alpha_{-l} -
 \alpha_{-k-l}\,  \alpha_{k} \, \alpha_{l} \Big] \,. $$

Since we are interested  now in $\widehat{\mathsf{S}}$, we can simplify the
differential equation:
\begin{equation} \label{xxqq2}
s_3 q\frac{d}{dq} \widehat{\mathsf{S}}
 =  (s_1+s_2)\mathsf{\widehat{A}} + [\mathsf{B}, \widehat{\mathsf S}]  \ \ \ \mod (s_1+s_2)^2\ .
\end{equation}
for
$$\widehat{\mathsf{A}} = \left(\sum_{k>0} k \frac{(-q)^k}{(-q)^k-1} \,
 \alpha_{-k} \, \alpha_k \right) \circ \mathsf{S}(0).$$

The eigenvectors for ${\mathsf{B}}$ (mod $s_1+s_2$) are the classes
 $\JJ_\lambda$ with eigenvalues 
\[ %doublecheck sign
w_\lambda = \sum_{(i,j)\in\lambda}(i-j)s_1.
\]
Equation \eqref{xxqq2} then
gives a simple relationship between the entries of 
$\widehat{\mathsf{S}}$ and of $\widehat{\mathsf{A}}$ in the $\JJ_\lambda$ basis.

The
operator $\widehat{\mathsf{A}}$ is diagonal in the Nakajima basis $\CC_\lambda$
 with entries
\[
\widehat{\mathsf{A}}^{\CC}_{\lambda\lambda} = 
\sum_{k \text{ part of }\lambda}k^2\frac{(-q)^k}{(-q)^k-1}.
\]
Applying the change of basis formula \eqref{changebasis}, we obtain 
the entries in the $\JJ_\lambda$ basis (mod $s_1+s_2$):
\[ %change this z to a better z
\widehat{\mathsf{A}}^{\JJ}_{\mu'\mu} 
= \sum_{\lambda}\frac{\dim\mu'}{\dim\mu}
\frac{\chi^\mu(\lambda)\chi^{\mu'}(\lambda)}
{z(\lambda)}\sum_{k \text{ part of }\lambda}k^2\frac{(-q)^k}{(-q)^k-1}. 
\]
If we use the notation $\lambda_r$ for the number of parts in a partition $\lambda$ of size $r$, then we can rewrite the entries as:
\[ 
\widehat{\mathsf{A}}^{\JJ}_{\mu'\mu} = \frac{\dim\mu'}{\dim\mu}\sum_{r=1}^dr\frac{(-q)^r}{(-q)^r-1}\sum_{\lambda}\frac{\chi^\mu(\lambda)\chi^{\mu'}(\lambda)r\lambda_r}{z(\lambda)}. 
\]
The following Lemma (easily proven using the Murnaghan-Nakayama rule) gives a simpler expression for the innermost sum in the above expression.
\begin{lem}\label{varpi}
Let $\mu,\mu'$ be partitions of the same size and let $r>0$. Then
\[
\sum_{\lambda}\frac{\chi^\mu(\lambda)\chi^{\mu'}(\lambda)r\lambda_r}{z(\lambda)} = \sum_{\substack{\gamma,\gamma'\text{ $r$-hooks} \\ \mu\backslash\gamma = \mu'\backslash\gamma'}}(-1)^{h(\gamma)+h(\gamma')}
\]
where $h(\gamma)$ is the number of rows in a rim hook $\gamma$.
\end{lem}
\noindent In the calculations below, 
we denote by $\theta_r(\mu,\mu')$ the quantity appearing in Lemma~\ref{varpi}.

We now are able to compute the restriction
$$\frac{1}{s_1+s_2}\widehat{\mathsf{S}}^{\mu}_{(d)}\Big|_{s_2=-s_1} $$
in terms of 
$\theta_r$ and the eigenvalues $w(\lambda)$: 
\footnotesize
\begin{multline*} 
%\frac{1}{s_1+s_2}\widehat{\mathsf{S}}^{\mu}_{(d)}\Big|_{s_2=-s_1} 
%= \\
\sum_{r=1}^d\left(\frac{(-1)^{d-1}s_1^{d-1}(d-1)!}{\dim\mu}\sum_\nu \chi^\nu((d))
\theta_r(\mu,\nu)\frac{r}{w(\mu)-w(\nu)+n_{\mu,\nu} s_3}\right)\frac{(-q)^r}{1-(-q)^r}.
\end{multline*}
\normalsize
As before, $\chi^\nu((d)) = 0$ unless $\nu = \alpha_a$ is a hook. For 
$\theta_r(\mu,\nu)$ to be nonzero, we must have $\mu$ be the union of two hooks, of sizes $d-r$ and $r$. 
The integers $n_{\mu,\nu}$ which arise will not affect the answer.

If we multiply by the descendent and edge factors and take the constant term in $\frac{s_3}{s_1}$, we obtain an expression for \eqref{hhthh} of the form
\[
\sum_{r=1}^d B_r\frac{(-q)^r}{1-(-q)^r}.
\]
The explicit formulas for $B_r$ depend upon two cases.
For $B_d$, we have
\footnotesize
\begin{multline*} 
(-1)^d d \cdot d!(d+2)!\ B_d =\\ 
\sum_{a+b=d-1}(-1)^a\binom{d-1}{a}\sum_{\substack{i=-b \\ i\ne 0}}^a\frac{1}{i}\left(-(-b-1)^{d+2}+(-b)^{d+2}+a^{d+2}-(a+1)^{d+2}\right)\ ,
\end{multline*}
\normalsize
and for $r<d$, we have
\footnotesize
\begin{multline*} 
\frac{(-1)^{d+r}}{2} d\cdot d!(d+2)!\ B_r =\\
\sum_{\substack{a+b=d-1 \\ a\ge r \\ 0\le c\le r-1}}(-1)^{a+c}\binom{d-r-1}{a-r}\binom{r-1}{c}\frac{(a-r-c)(b+c+1-r)}{(a-c)^2(b+c+1)} 
\\ \cdot
\Bigg(-(-b-1)^{d+2}+(-b)^{d+2}+(a-r)^{d+2}-(a-r+1)^{d+2}
\\ -(c-r)^{d+2}+(c-r+1)^{d+2}+c^{d+2}-(c+1)^{d+2}\Bigg)\ .
\end{multline*}
\normalsize 

\subsection{Classical pairing}
We compute 
the classical pairing $\langle \tau_d,\CC_{(d[0])}\rangle$. 
 The simpler pairing $\langle \tau_{d-1},\CC_{(d[0])}\rangle$ 
is needed for the calculation and is addressed
first. 
\begin{lem}\label{dminusone}
$\langle \tau_{d-1},\CC_{(d[0])}\rangle = \frac{1}{d!}$
\end{lem}
\begin{proof}
By dimension counting, the pairing has no dependence on $s_1$ and $s_2$, 
so we can work mod $s_1+s_2$. Localization then yields
\footnotesize
\begin{multline*}
(-1)^d {d\cdot d!(d+1)!}\
\Big\langle \tau_{d-1},\CC_{(d[0])}\Big\rangle 
\\ = \sum_{a+b=d-1}(-1)^a\binom{d-1}{a}\left(-(-b-1)^{d+1}+(-b)^{d+1}+a^{d+1}-(a+1)^{d+1}\right)
\end{multline*}
\normalsize
If we rewrite $-(-b-1)^{d+1}+(-b)^{d+1}+a^{d+1}-(a+1)^{d+1}$ as a polynomial in $a$ alone, the leading term $-(d+1)d^2a^{d-1}$. Then,
\[
\Big\langle \tau_{d-1},\CC_{(d[0])}\Big\rangle = \frac{(-1)^d}{d\cdot d!(d+1)!}(-(d+1)d^2)(-1)^{d-1}(d-1)! = \frac{1}{d!}
\]
since the contributions of all the lower terms are 0.
\end{proof}

We cannot compute $\langle \tau_d,\CC_{(d[0])}\rangle$ in the same way, since we cannot work mod $s_1+s_2$ (as we know the answer is a multiple of $s_1+s_2$). 
Instead we work mod $s_2$ and consider the function
\[
f(k) = (k+1)!s_1^{d-1-k}\Big\langle \tau_k,\CC_{(d[0])}\Big\rangle\Big|_{s_2=0}\ .
\]
We can compute by localization that $f$ is of the form
\[
\sum_{i=1}^dc_ii^{k+1}
\] 
for some constants $c_i\in\Q$. We also know 
$$f(0) = f(1) = \cdots = f(d-2) = 0$$
by dimension constraints. By Lemma~\ref{dminusone}, we have $f(d-1)=1$. 
Interpolation then gives $f(d)=\frac{d(d+1)}{2}$. We conclude the
following result.
\begin{prop}\label{classicald}
\[
\Big\langle \tau_{d},\CC_{(d[0])}\Big\rangle = \frac{s_1+s_2}{2\cdot(d-1)!}.
\]
\end{prop}

\subsection{Proof of Theorem \ref{ytytyt}}
\label{hhttyy}
We have
\[
F(d) = F_0(d) + \sum_{r=1}^d (A_r+B_r)\frac{(-q)^r}{1-(-q)^r}\ .
\]
Although the formulas
for $A_r$ and $B_r$ calculated in Sections \ref{vcv} - \ref{rcr} are 
 very complicated, a wonderful  
combinatorial identity holds:
\begin{equation} \label{gthh3}
A_r+B_r = \frac{1}{d!}
\end{equation}
for all $1\leq r \leq d$. The proof of \eqref{gthh3} is by straightforward 
manipulation using a few standard binomial sum identities. In the case $r = d$, all that is needed is the identity
\[
\sum_{i = 0}^m (-1)^i\binom{m}{i}(a_mi^m + a_{m-1}i^{m-1}+ \cdots + a_0) = (-1)^mm!\cdot a_m.
\]
For $r < d$, the following two identities must also be used to compute the sum over $c$ in the expression for $B_r$:
\[
\sum_{i=0}^m(-1)^i\binom{m}{i}\frac{1}{x+i} = \frac{1}{(m+1)\binom{x+m}{m+1}}
\]
\[
\sum_{i=0}^m(-1)^i\binom{m}{i}\frac{1}{(x+i)^2} = \frac{1}{(m+1)\binom{x+m}{m+1}}\sum_{i=0}^m\frac{1}{x+i}.
\]

Combined with Proposition~\ref{classicald}, the identity \eqref{gthh3} yields
\[
F(d) = \frac{1}{2\cdot d!}\sum_{r=1}^d \frac{1+(-q)^r}{1-(-q)^r},
\]
which completes the proof of Theorem~\ref{ytytyt}. \qed

The same method of computation actually yields a relatively simple formula for a larger family of invariants.
Suppose that $m_1,\ldots,m_k$ are positive integers satisfying
$$\sum_{i=1}^k m_i = d\ .$$
Lemma \ref{pwpw} relied only on a dimension analysis which
 also applies to $\mathsf{Z}^{\mathsf{cap}}_{d,(d)}(\tau_{m_1}([0])\cdots\tau_{m_k}([0]))^{\T}$, so we can expect the
 series to be relatively simple. In fact, we can prove
\begin{multline}\label{pr44}
\mathsf{Z}^{\mathsf{cap}}_{d,(d)}(\tau_{m_1}([0])\cdots\tau_{m_k}([0]))^{\T} = \\
\frac{q^d}{m_1!\cdots m_k!}\left(\frac{s_1+s_2}{s_1s_2}\right)\frac{1}{2}\sum_{r=1}^dC_r(m_1,\ldots,m_k)\frac{1+(-q)^r}{1-(-q)^r}
\end{multline}
for coefficients
\[
C_r(m_1,\ldots,m_k) = \sum_{\substack{I \subset \{1,\ldots,k\} \\ \sum_{i\in I}m_i < r}}r^{|I|-1}(d-r)^{k-|I|-1}\left(r - \sum_{i\in I}m_i\right).
\]
For example, we have
\begin{multline*}
\mathsf{Z}^{\mathsf{cap}}_{3,(3)}(\tau_{1}([0])\tau_{2}([0]))^{\T} = \\
\frac{q^3}{2}\left(\frac{s_1+s_2}{s_1s_2}\right)\frac{1}{2}\left(2\cdot\frac{1-q}{1+q}+2\cdot\frac{1+q^2}{1-q^2}+3\cdot\frac{1-q^3}{1+q^3}\right).
\end{multline*}

When proving \eqref{pr44} by the method used for Theorem~\ref{ytytyt}, 
the leading classical term requires the 
somewhat surprising identity
\[
\sum_{r=1}^d C_r(m_1,\ldots,m_k) = \sum_{\substack{f:\{1,\ldots,k\}\to\{1,\ldots,k\} \\ \text{$f$ has only one periodic orbit}}}\prod_{i=1}^k m_{f(i)}\ .
\]
We can prove this identity by showing that both sides satisfy the same recurrence equation as a function of $m_1,\ldots,m_k$:
\begin{align*}
F(m_1,\ldots,m_k) = m_k^2(m_1+\cdots+m_k)^{k-2} &+ m_1F(m_1+m_k,m_2,\ldots,m_{k-1}) \\ &+ m_2F(m_1,m_2+m_k,\ldots,m_{k-1}) \\ &+ \cdots \\ &+ m_{k-1}F(m_1,m_2,\ldots,m_{k-1}+m_k).
\end{align*}

Finally, we can also use the same method of computation to obtain the analogous normalized Donaldson-Thomas partition functions. Surprisingly, these are equal to the stable pairs partition functions except in degree one:
\begin{equation*}
\mathsf{Z}^{\mathsf{DT,cap}}_{d,(d)}( \tau_d(\mathsf{p}))^T = \begin{cases}
\mathsf{Z}^{\mathsf{cap}}_{d,(d)}( \tau_d(\mathsf{p}))^T &\text{if }d > 1 \\
\mathsf{Z}^{\mathsf{cap}}_{d,(d)}( \tau_d(\mathsf{p}))^T + q\left(\frac{s_1+s_2}{s_1s_2}\right)q\frac{d}{dq}\log(M(-q))&\text{if }d = 1. \end{cases}
\end{equation*}
Here $M(q) = \prod(1-q^r)^{-r}$ is the MacMahon function.

\vspace{+16 pt}
\noindent Departement Mathematik \hfill Department of Mathematics \\
\noindent ETH Z\"urich \hfill  Princeton University \\
\noindent rahul@math.ethz.ch  \hfill rahulp@math.princeton.edu \\

\vspace{+8 pt}
\noindent
Department of Mathematics\\
Princeton University\\
apixton@math.princeton.edu


\begin{thebibliography}{MNOP2}



\bibitem{Beh} K. Behrend, {\em Donaldson-Thomas invariants via
microlocal geometry}, math.AG/0507523.



\bibitem{bridge} T. Bridgeland, {\em Hall algebras and curve-counting
invariants}, arXiv:1002.4372.




\bibitem{BryanP}
J.~Bryan and R.~Pandharipande,
\newblock {\em The local {G}romov-{W}itten theory of curves}, 
JAMS {\bf 21} (2008), 101--136.
\newblock math.AG/0411037.



\bibitem{DT}
S.~K. Donaldson and R.~P. Thomas,
\newblock {\em Gauge theory in higher dimensions}.
\newblock In {\em The geometric universe (Oxford, 1996)},  31--47. Oxford Univ.
  Press, Oxford, 1998.

\bibitem{FP} C. Faber and R. Pandharipande, {\em Hodge integrals and
Gromov-Witten theory}, Invent. Math. {\bf 139} (2000), 173--199.


\bibitem{GraberP}
T.~Graber and R.~Pandharipande,
\newblock {\em Localization of virtual classes}, Invent. Math., {\bf 135},
  487--518, 1999.
\newblock math.AG/9708001.


\bibitem{HLShaves}
D.~Huybrechts and M.~Lehn,
\newblock {\em The geometry of moduli spaces of shaves}.
\newblock Aspects of Mathematics, E31. Friedr. Vieweg \& Sohn, Braunschweig,
  1997.
%\newblock Out of print, available at
%  http:$/\!/$www.math.uni-bonn.de/people/huybrech/moduli.ps.


\bibitem{joy} D. Joyce and Y. Song, {\em A theory of generalized
Donaldson-Thomas invariants}, arXiv:0810.5645.



\bibitem{LPPairs1}
J.~Le~Potier,
\newblock {\em Syst\`emes coh\'erents et structures de niveau}, Ast\'erisque,
  {\bf 214}, 1993.


\bibitem{liwu} J. Li and B. Wu, {\em Degeneration of Donaldson-Thomas
invariants}, preprint 2009.





\bibitem{moop}
D.~Maulik, A. ~Oblomkov, A.~Okounkov, and R.~Pandharipande,
\newblock {\em The Gromov-{W}itten/{D}onaldson-{T}homas correspondence
for toric 3-folds}, Invent. Math. (to appear).
arXiv:0809.3976.

\bibitem{MNOP1}
D.~Maulik, N.~Nekrasov, A.~Okounkov, and R.~Pandharipande,
\newblock {\em Gromov-{W}itten theory and {D}onaldson-{T}homas theory. {I}},
  Compos. Math. {\bf 142} (2006), 1263--1285.


\bibitem{MNOP2}
D.~Maulik, N.~Nekrasov, A.~Okounkov, and R.~Pandharipande,
\newblock {\em Gromov-{W}itten theory and {D}onaldson-{T}homas theory. {II}},
  Compos. Math. {\bf 142} (2006), 1286--1304.

\bibitem{mptop}
D.~Maulik and R.~Pandharipande,
\newblock {\em A topological view of Gromov-Witten theory},
  Topology {\bf 45} (2006), 887--918.




\bibitem{mpt}
D. Maulik, R. Pandharipande, R. Thomas, {\em Curves on $K3$ surfaces
and modular forms}, J. of Topology {\bf 3} (2010), 937--996.
arXiv:1001.2719.


\bibitem{vir}
A.~Okounkov and R.~Pandharipande,
\newblock {\em Virsoro constraints for target curves},
Invent. Math. {\bf 163} (2006), 47--108.

\bibitem{hilb1}
A.~Okounkov and R.~Pandharipande,
\newblock {\em The quantum cohomology of the Hilbert
scheme of points of the plane},
Invent. Math. {\bf 179} (2010), 523--557.



\bibitem{lcdt}
A.~Okounkov and R.~Pandharipande,
\newblock {\em The local Donaldson-Thomas theory of curves},
Geom. Topol. {\bf 14} (2010), 1503--1567.
\newblock math.AG/0512573.

\bibitem{hilb2}
A.~Okounkov and R.~Pandharipande,
\newblock {\em The quantum differential equation of the Hilbert
scheme of points of the plane}, Transform. Groups {\bf 15} (2010), 965--982.
\newblock arXiv/09063587.


\bibitem{partone}
R.~Pandharipande and A.~Pixton,
\newblock {\em Descendents on local curves: rationality}, 
arXiv:10114050.

\bibitem{part3}
R.~Pandharipande and A.~Pixton,
\newblock {\em Descendent theory of stable pairs on toric 3-folds}, 
arXiv:10114054.

\bibitem{pt}
R.~Pandharipande and R.~P. Thomas,
\newblock {\em Curve counting via stable pairs in the derived
category}, Invent Math. {\bf 178} (2009), 407 -- 447.


\bibitem{pt2}
R.~Pandharipande and R.~P. Thomas,
\newblock {\em The 3-fold vertex via stable pairs}, Geom Topol.
{\bf 13} (2009), 1835--1876.

\bibitem{pt3}
R.~Pandharipande and R.~P. Thomas,
\newblock {\em Stable pairs and BPS invariants}, JAMS {\bf 23}
(2010), 267--297.



\bibitem{Thomas}
R.~P. Thomas,
\newblock {\em A holomorphic {C}asson invariant for {C}alabi-{Y}au 3-folds, and
  bundles on {$K3$} fibrations}, J. Differential Geom. 
{\bf 54}, 367--438,
  2000.


\bibitem{toda}
Y. Toda, {\em Generating functions of stable pairs invariants
via wall-crossings in derived categories}, arXiv:0806.0062.








\end{thebibliography}
\end{document}